\documentclass[letterpaper,12pt,oneside]{article}
\usepackage[T1]{fontenc}
\usepackage[latin1]{inputenc}
\usepackage{epsfig}
\usepackage{amsmath}
\usepackage{amssymb}
\usepackage{amsthm}
\usepackage[active]{srcltx}
\usepackage{dsfont}
\usepackage[titles]{tocloft}
\usepackage{color}
\numberwithin{equation}{section}
\newtheorem{theoreme}{Theorem}[section]
\newtheorem{proposition}{Proposition}[section]

\newtheorem{definition}[theoreme]{Definition}

\newcommand{\RR}{\ensuremath{\mathbb R}}

\newcommand{\EE}{\ensuremath{\mathbb E}}
\newcommand{\NN}{\ensuremath{\mathbb N}}
\newcommand{\I}{\ensuremath{\mathcal I}}
\newcommand{\LL}{\ensuremath{\mathcal L}}
\newcommand{\J}{\ensuremath{\mathcal J}}
\newcommand{\M}{\ensuremath{\mathcal M}}

\newcommand{\xx}{\ensuremath{\mathbf{x}}}
\newcommand{\XX}{\ensuremath{\mathbf{X}}}
\newcommand{\ddd}{\ensuremath{\mathbf{d}}}

\newcommand{\om}{\omega}
\newcommand{\Om}{\Omega}
\newcommand{\la}{\lambda}

\newcommand{\De}{\Delta}
\newcommand{\ep}{\varepsilon}

\newcommand{\ga}{\gamma}
\newcommand{\Ga}{\Gamma}

\newcommand{\limla}{\lim_{\lambda \to 0}}
\newcommand{\ccc}{\ensuremath{\mathbf{c}}}

\newcommand{\limn}{\lim_{n\to\infty}}

\title{The asymptotic value in finite stochastic games}

\author{Miquel Oliu-Barton \footnote{
Combinatoire et Optimisation, IMJ, CNRS UMR 7586, Facult\'e de Math\'ematiques, Universit\'e P. et M. Curie
Paris 6, Tour 15-16, 1er etage, 4 Place Jussieu, 75005 Paris and Universit\'e d'Orsay Paris-Sud https://sites.google.com/site/oliubarton}}
\date{November 19, 2012}

\begin{document}
\maketitle

\abstract{We provide a direct, elementary proof for the existence of $\limla v_\la$, where $v_\la$ is the value 
of $\la$-discounted finite two-person zero-sum stochastic game. 
}

\section{Introduction}\label{model}
Two-person zero-sum stochastic games were introduced by Shapley \cite{shapley53}.
 They are described by a $5$-tuple $(\Om, \I,\J,q,g)$, where $\Om$ is a finite set of states, $\I$ and $\J$ are finite sets of actions,
$g:\Om \times \I \times \J \to [0,1]$ is the payoff, $q: \Om \times \I \times \J \to \De(\Om)$ the transition and,
 for any finite set $X$, $\De(X)$ denotes the set of probability distributions over $X$.  The functions
$g$ and $q$ are bilinearly extended to $\Om\times \De(\I)\times \De(\J)$.
 The stochastic game with initial state $\om\in\Om$ and discount factor $\la\in(0,1]$ is denoted by 
$\Ga_\la(\om)$ and is played as follows:  at stage $m\geq1$, knowing the current state $\om_m$, the players choose actions  
$(i_m,j_m)\in \I\times \J$; their choice produces a stage payoff $g(\om_m,i_m,j_m)$ and influences the transition: a new state $\om_{m+1}$ is chosen according to the probability distribution $q(\cdot|\om_m,i_m,j_m)$.
At the end of the game, player $1$ receives $\sum\nolimits_{m \geq 1}\la(1-\la)^{m-1} g(\omega_m,i_m,j_m)$ from player $2$.
The game $\Ga_\la(\om)$ has a value $v_\la(\om)$, and $v_\la=(v_\la(\om))_{\om\in\Om}$ is the unique fixed point of the so-called Shapley operator \cite{shapley53}, i.e. $v_\la=\Phi(\la,v_\la)$, where for all 
$f\in \RR^\Om$:
\begin{equation}\label{shapley} \Phi(\la,f)(\om)=\mathrm{val}_{(s,t)\in \De(\I)\times \De(\J)}
\{\la g(\om,s,t)+(1-\la)\EE_{q(\cdot|\om,s,t)}[f(\widetilde{\om})]\}.
\end{equation}
The Shapley operator provides optimal stationary strategies for both players. In particular, the result holds 
for any signalling structure on past actions.
 The existence of $\limla v_\la$ was established by Bewley and Kohlberg \cite{BK76}, using
Tarski-Seidenberg elimination theorem. \\
\indent The purpose of this note is to provide a direct, self-contained proof for the existence of $\limla v_\la$. The key idea is to represent the asymptotic behaviour of a sequence of strategies by a simpler object.
Let $(x,y)\in\De(\I)^\Om\times \De(\J)^\Om$ be a pair of stationary strategies. 
Every time the state $\om\in \Om$ is reached the next state is distributed according to $q(\cdot|\om,x(\om),y(\om))$ and the stage payoff is $g(\om,x(\om),y(\om))$. Thus, 
the sequence of states $(\om_m)_m$ is a Markov chain with transition $Q=(q(\om'|\om,x(\om),y(\om))_{(\om,\om')\in \Om^2}$ and the stage payoffs can be described by  a vector 
$g=(g(\om,x(\om),y(\om))_{\om\in \Om}$.
For any initial state $\om$, the 
expected payoff induced by $(x,y)$ in $\Ga_\la(\om)$ is given by
$$\ga_\la(\om,x,y)=\sum\nolimits_{\om'\in\Om} t_\la(\om,\om')g(\om'),$$
where $t_\la(\om,\om')=\sum_{m\geq 1} \la (1-\la)^{m-1}Q^{m-1}(\om,\om')$ is the mean $\la$-discounted time spent in state $\om'$.

A key observation, due to Solan \cite{solan03}, is that $t_\la(\om,\om')$ can be written has a hitting time of an auxiliary Markov chain 
whose transitions are in the set $\{0, \la, ((1-\la) Q(\om,\om'))_{(\om,\om')\in \Om^2}\}$. Thus, 
using a classical result from Friedlin and Wentzell for finite Markov chains, one deduces that $t_\la(\om,\om')$ is a rational fraction in the variables $\la$ and $((1-\la) Q(\om,\om'))_{(\om,\om')\in \Om^2}$, and that both polynomials in the numerator and denominator have nonnegative coefficients and are of degree at most $|\Om|$. For a fixed $y$, a similar assertion is obtained for $\ga_\la(\om,x,y)$ as a function of the variables $\la$ and $((1-\la) x^i(\om))_{(\om,i)\in \Om\times \I}$. That is, $\ga_\la(\om,x,y)$ is a rational fraction in these variables. One can easily check that the monomials both in the numerator and denominator can then be written in the following form:
\begin{equation}\label{ds}
C(1-\la)^b\la^a \prod_{(\om,i)\in\Om\times \I} x^{i}(\om)^{A(\om,i)},
\end{equation}
where $C>0$ depends on $(y,\om)$ but not on $(x,\la)$, $a,b\in \{0,\dots, |\Om|\}$ and $A\in \{0,1\}^{\Om\times \I}$.
\subsection{The asymptotic payoff} 
\label{lim}
Consider now a sequence $(\la_n,x_n)_n$, where $\la_n\in (0,1]$ is a discount factor and $x_n\in \De(\I)^\Om$ is a stationary strategy, for all $n\in \NN$. 
$\ga_{\la_n}(\om,x_n,y)$, as $n$ tends to infinity, for a fixed stationary strategy $y\in \De(\J)^\Om$. 
\begin{definition}  A sequence $(\la_n,x_n)_n$ in $(0,1]\times \De(\I)^\Om$ is \emph{regular} if $\limn \la_n=0$ and if for any two monomials of the form \eqref{ds} their ratio converges in $[0,+\infty]$ as $n$ tends to infinity.\footnote{We use here the natural convention that 
$\frac{0}{0}=0^0=1$ and $0^\beta=0$, $0^{-\beta}=\frac{\beta}{0}=+\infty$, for all $\beta>0$.} 
\end{definition}
Regular sequences can be characterized by a vector. Indeed, introduce a finite set:
$$\M:=
\{(A,a)\ | \ A \in \{-1,0,1\}^{\Om\times \I}, \ a \in \{-|\Om|,\dots,0,\dots, |\Om|\}\}.$$
The sequence $(\la_n,x_n)_n$ is regular if for all $(A,a)\in \M$ the following limit
\begin{equation}\label{defl}
L[(\la_n,x_n)_n](A,a):=\limn  \la_n^{a} \prod_{(\om,i)\in \Om\times \I} x_n^i(\om)^{A(\om,i)}
\end{equation}
exists in $[0,+\infty]$.
The regularity of a sequence depends on the existence of finitely many  limits. Thus, for any family $(x_\la)_{\la\in(0,1]}$ of stationary strategies there exists $(\la_n)_n$ such that $(\la_n,x_{\la_n})_n$ is regular. 

\begin{proposition}\label{algebra} Let $y\in \De(\J)^\Om$ and $\om\in \Om$ be fixed. For any regular sequence 
$(\la_n,x_n)_n$ , $\limn \ga_{\la_n}(\om,x_n,y)$
 exists 
 and depends only on the vector $L[(\la_n,x_n)_n]$.
\end{proposition}
\begin{proof}  Let $(\la_n,x_n)_n$ be regular and let $L=L[(\la_n,x_n)_n]$.
 We have already seen
that the expected payoff induced by $(x_n,y)$ in $\Ga_{\la_n}(\om)$ can be written as a rational fraction whose
monomials are all of the form: 
\begin{equation}\label{ds2}
m_n:=C(1-\la_n)^b\la_n^a \prod_{(\om,i)\in\Om\times \I} x_n^{i}(\om)^{A(\om,i)},
\end{equation}
that the ratio of any two monomials $m_n$ and $m'_n$ converges as $n\to \infty$, and  that the limit is determined by $L$ (and the constants $C,C'>0$). Thus, one can use the vector $L$ to define an order relation in the set of the monomials in $\ga_{\la_n}(\om,x_n,y)$ as follows: $m_n\preceq m'_n$ if and only if
$\lim_{n\to \infty} m_n  /m'_n\in[0,+\infty)$.
The set is totally ordered. 
Dividing numerator and denominator by some maximal element $m^*_n$, and 
taking $n\to \infty$ we obtain that:
\begin{equation}\label{fra2}\limn \ga_{\la_n}(\om,x_n,y)=
 \frac{\sum_{(A,a)\in \M^+} C(A,a)L(A-A^*,a-a^*)}{\sum_{(A,a)\in \M^+} C'(A,a)  L(A-A^*,a-a^*)},
\end{equation}
where $\M^+:=\{ (A,a)\ | \ A\in \{0,1\}^{\Om \times \I}, \ a\in \{0,\dots, |\Om|\}\}$, and where the constants $C(A,a)$ and $C'(A,a)$ are nonnegative for all $(A,a)\in \M^+$.
The maximality of $m^*$ ensures that $L(A-A^*,a-a^*)\in [0,+\infty)$, for all $(A,a)\in \M^+$ 
and that
not all are $0$. 
The result follows. 
 $\square$ \end{proof}

\subsection{Canonical strategies}
For any $\ccc=(\ccc(\om,i))$ and $\mathbf{e}=(\mathbf{e}(\om,i))$ in $\RR_+^{\Om\times I}$,  we define a family of stationary strategies $(\xx_\la)_\la$ as follows:
\begin{equation}\label{def}\xx_\la^i(\om):=\frac{\ccc(\om,i)\la^{\mathbf{e}(\om,i)}}{\sum_{i'\in\I}\ccc(\om,i')\la^{\mathbf{e}(\om,i')}}, \quad \forall (\om,i)\in \Om\times\I, \ \forall \la\in(0,1]. 
\end{equation}
Assume, in addition, that
$\sum\nolimits_{i\in \I, \ \mathbf{e}
(\om,i)=0}\ccc(\om,i)=1$ for all $\om$, so that
\begin{equation}\label{can}
\xx_\la^i(\om)\sim_{\la\to 0}\ccc(\om,i)\la^{\mathbf{e}(\om,i)}, \quad \forall (\om,i)\in \Om\times \I.
\end{equation}
The exponent determines the 
order of magnitude of the probability of playing the action $i$ at state $\om$ asymptotically; the coefficient $\ccc(\om,i)$ its intensity. 
\begin{definition} A family of strategies $(\xx_\la)_{\la\in(0,1]}$ is \emph{canonical} if it is induced by some $\xx=(\ccc,\mathbf{e})$ in the following set:
$$\XX=\{(\ccc,\mathbf{e})\in (\RR^*_+\times \RR_+)^{\Om\times \I} \ | \ \forall \om \in \Om, \ \sum\nolimits_{i\in \I, \ \mathbf{e}(\om,i)=0}\ccc(\om,i)=1\}.$$
Note that the coefficients are taken strictly positive.
\end{definition}
For all
$(A,a)\in \M$ and $\xx=(\ccc,\mathbf{e})\in \XX$ the following limit exists:
\begin{equation}\label{lx2}L_\xx(A,a):=\limla \la^a\prod\nolimits_{(\om,i)}\xx_\la^i(\om)^{A(\om,i)}.\end{equation}
Indeed, a direct consequence of \eqref{can} is that: $$L_\xx(A,a)=\limla \la^{a+\sum\nolimits_{(\om,i)}A(\om,i)\mathbf{e}(\om,i)} \prod\nolimits_{(\om,i)}\ccc(\om,i)^{A(\om,i)},$$
where 
$\prod\nolimits_{(\om,i)}\ccc(\om,i)^{A(\om,i)}>0$. Thus:
\begin{equation}
\label{lx} L_\xx(A,a)\in 
\begin{cases} \{0\},& \text{ iff }a+\sum\nolimits_{(\om,i)}A(\om,i)\mathbf{e}(\om,i)>0, \\
\{+\infty\},& \text{ iff }a+\sum\nolimits_{(\om,i)}A(\om,i)\mathbf{e}(\om,i)<0,\\
(0,+\infty),& \text{ iff }a+\sum\nolimits_{(\om,i)}A(\om,i)\mathbf{e}(\om,i)=0.
\end{cases}
\end{equation}
Thus, for any $\xx\in \XX$ and any vanishing sequence $(\la_n)_n$ of discount factors, the sequence $(\la_n,\xx_{\la_n})_n$ is regular.  
Moreover, 
$L_\xx=L[(\la_n, \xx_{\la_n})_n]$ for any such sequence. 
\section{Main results}
\subsection{Representation of a regular sequence by a canonical strategy} 
Fix some regular sequence $(\la_n,x_n)_n$ throughout this section and let $L=L[(\la_n,x_n)_n]\in [0,+\infty]^\M$ the vector defined in \eqref{defl}. Notice that $L$ has many elementary properties:
\begin{itemize} 
 \item [$(P1)$] $L(0,0)=1$ and, for all $(A,a)\neq 0$, $L(A,a)=+\infty$ if and only  if $L(-A,-a)=0$;
 \item [$(P2)$] For all $\mu \in \RR$, $L(0,\mu):=\lim\nolimits_{n\to \infty} \la_n^{\mu}=0 \Leftrightarrow \mu>0$ and 
 $L(0,\mu)\in (0,+\infty) \Leftrightarrow \mu=0$. In particular, $L(0,\mu)\in \{0,1,+\infty\}$ for all $\mu\in \RR$;
\item[$(P3)$] If $L(A,a)<+\infty$, $L(\mu A,\mu a):=\limn \la_n^{\mu a}\prod\nolimits_{(\om,i)} x_n^i(\om)^{\mu A(\om,i)}=L(A,a)^\mu $;
\item[$(P4)$] If $L(A,a)<+\infty$ and $L(B,b)<+\infty$, then 
$L(A+B,a+b)=L(A,a)L(B,b)$.
\end{itemize}

\begin{proposition}\label{elputolema}  
There exists $\xx \in \XX$ such that $L_\xx=L$. 
 \end{proposition}   
\begin{proof} 
Note that $\prod\nolimits_{(\om,i)}\ccc(\om,i)^{A(\om,i)}>0$ for any $A\in \{-1,0,1\}^{\Om\times I}$. Thus, from \eqref{lx} and $(P1)$ one deduces the following necessary and sufficient conditions on the coefficients and the exponents $(\ccc,\mathbf{e})$ of $\xx$ 
 for having $L_\xx=L$:
\begin{eqnarray}\label{ns1}
\sum\nolimits_{(\om,i)}A(\om,i)\mathbf{e}(\om,i)+a>0,& \forall (A,a)\in \M \text{ s.t. } L(A,a)=0
,\\ \label{ns3}
\sum\nolimits_{(\om,i)}A(\om,i)\mathbf{e}(\om,i)+a=0,&\forall (A,a)\in \M \text{ s.t. }L(A,a)\in(0,+\infty),\\
 \label{ns2}
\prod\nolimits_{(\om,i)}\ccc(\om,i)^{A(\om,i)}=L(A,a),& \forall (A,a)\in \M \text{ s.t. }L(A,a)\in(0,+\infty). 
\end{eqnarray}
\textbf{Notation:} Let $\LL_0:=\{ (A,a) \in \M \ | \ L(A,a)=0\}$ and $\LL_+:=\{(A,a) \in \M \ | \ L(A,a)\in (0,+\infty)\}$. Put $\LL:=\LL_0\cup \LL_+ $.\\
\textbf{Solving for the exponents.}
Let us prove that the system \eqref{ns1}-\eqref{ns3}
has a solution.
One and only one of the systems \eqref{ns1}-\eqref{ns3} and \eqref{a3}-\eqref{a4}-\eqref{a5} is consistent
(see Mertens, Sorin and Zamir \cite{MSZ94}, part A, page 28):
\begin{eqnarray}
\sum\nolimits_{(A,a)\in \LL} \mu(A,a) A=0, \quad \mu_{|\LL_0} 
\geq 0 , \label{a3}\\ 
-\sum\nolimits_{(A,a)\in \LL} \mu(A,a)   a \geq 0,\label{a4}\\
-\sum\nolimits_{(A,a) \in \LL} \mu(A,a)  a+\sum\nolimits_{(A,a) \in \LL_0} \mu(A,a) >0,\label{a5}
\end{eqnarray}
Let us prove that the system \eqref{a3}-\eqref{a4}-\eqref{a5}, with unknowns $\mu=(\mu(A,a))_{(A,a)} \in \RR^{\LL}$, is inconsistent. 
In \eqref{a3}, $\mu_{|\LL_0}:=(\mu(A,a))_{(A,a)\in \LL_0}$ denotes the restriction of $\mu$ to $\LL_0$.
Assume \eqref{a3}. On the one
 hand, by $(P3)$-$(P4)$, for all $\mu\in \RR^{\LL}$:
\begin{eqnarray}\label{b1}
 \prod_{(A,a) \in\LL_+ }L(A,a)^{\mu(A,a)}
=L\left(\sum\nolimits_{(A,a) \in \LL_+ } \mu(A,a) A,\sum\nolimits_{(A,a) \in \LL_+ } \mu(A,a) a\right)
\in (0,+\infty).
\end{eqnarray}
On the other hand, 
by $(P3)$-$(P4)$, for all $\mu\in \RR^\LL$ such that $\mu_{|\LL_0}\geq 0$ 
one has:
\begin{equation}\label{b2}
\prod_{(A,a) \in\LL_0}L(A,a)^{\mu(A,a)}=L\bigg(\sum_{(A,a) \in \LL_0} \mu(A,a) A,\sum_{(A,a)
  \in \LL_0} \mu(A,a) a\bigg) 
 =\begin{cases}
 1 & \text{ if } \mu_{|\LL_0}=0,\\ 
 0  & \text{ otherwise.} 
 \end{cases}
 \end{equation}
Multiplying \eqref{b1} and \eqref{b2} yields, by assumption \eqref{a3} : 
\begin{equation}\label{b3}
L\left(0,\sum\nolimits_{(A,a) \in \LL }\mu(A,a) a\right)
\in\begin{cases}
(0,+\infty) & \text{ if } \mu_{|\LL_0}= 0,\\ 
\{0\}& \text{ otherwise.} 
\end{cases}
 \end{equation}
By $(P2)$, the first case implies $\sum\nolimits_{(A,a) \in \LL }\mu(A,a) a=0$, which contradicts \eqref{a5}, and the second case implies
 $\sum\nolimits_{(A,a) \in \LL }\mu(A,a) a>0$, which contradicts \eqref{a4}. 
The system \eqref{a3}-\eqref{a4}-\eqref{a5} being inconsistent, the existence of a solution to \eqref{ns1}-\eqref{ns3} in $\RR^{\Om\times \I}$ follows. The boundedness of $x_n(\om,i)$ implies that $L((0,\dots, 1^{(\om,i)},\dots,0),0)<+\infty$, so that $\mathbf{e}(\om,i)\geq 0$ by \eqref{ns1} and \eqref{ns3}.

 \noindent \textbf{Solving for the coefficients.} 
Taking the logarithm in \eqref{ns2} yields: 
\begin{equation}\label{r1}
\sum \nolimits_{(\om,i)}A(\om,i)\ln \ccc(\om,i)=\ln(L(A,a)),\quad  \forall (A,a)\in \LL_+, 
\end{equation}
which is a linear system in $\ddd=(\ln \ccc(\om,i))_{(\om,i)}\in \RR^{\Om\times \I}$.
 As before, one and only one of the systems \eqref{r1} and \eqref{f1}
is consistent:
\begin{eqnarray}\label{f1}\sum\nolimits_{(A,a)  \in \LL_+ } \mu(A,a)  A =0,\quad \sum\nolimits_{(A,a)  \in \LL_+ } \mu(A,a)
 \ln(L(A,a)) >0.
\end{eqnarray}
Let us prove that the system \eqref{f1}, with unkowns $\mu=(\mu(A,a))_{(A,a)} \in \RR^{\LL_+}$, is inconsistent.
 Suppose that $\sum_{(A,a)  \in \LL_+ } \mu(A,a)  A =0$. 
 Then, by $(P3)$-$(P4)$: 
$$\prod\nolimits_{(A,a) \in \LL_+ } L(A,a)^{\mu(A,a) }=
L\left(0,\sum\nolimits_{(A,a) \in \LL_+ }\mu(A,a) a \right)\in (0,+\infty).$$ 
By $(P2)$, this implies $\sum\nolimits_{(A,a) \in \LL_+ }\mu(A,a) a =0$ and, a fortiori, $\prod\nolimits_{(A,a) \in \LL_+ } L(A,a)^{\mu(A,a) }=1$, 
so that \eqref{f1} fails. Consequently, there exists $\ccc=(\exp(\ddd(\om,i))\in (\RR^*_+)^{\Om\times I}$ satisfying \eqref{ns2}.
   \end{proof}
\subsection{Convergence of the discounted values} 
\begin{theoreme}\label{main} 
The limit of $(v_\la)_\la$, as $\la$ tends to $0$, exists. 
 Moreover, there exists $\xx\in \XX$ such that $(\xx_\la)_\la$ is \emph{asymptotically optimal}, i.e. for all $\ep>0$, there exists $\la_0\in (0,1]$ such that:
\[ \ga_\la(\om, \xx_\la,y)\geq \lim\nolimits_{\la\to 0} v_\la(\om)-\ep,\quad \forall \om\in \Om,\ \forall y\in \De(\J)^\Om, \ \forall \la\in (0,\la_0).\]
\end{theoreme}
\begin{proof}
 Let $\om\in \Om$ be fixed. Let $(x_\la)_{\la>0}$ be a family of optimal stationary strategies in $(\Ga_\la(\om))_{\la>0}$ and let $(\la_n)_{n}$ be a sequence of discount factors such that
$\limn v_{\la_n}(\om)=\limsup_{\la\to 0}v_\la(\om)$. The optimality of $x_{\la_n}$ implies that
$\ga_{\la_n}(\om,x_{\la_n},j)\geq v_{\la_n}(\om)$, for all $j\in \J^\Om$. Indeed, against a stationary strategy of 
player $1$, player $2$ faces a Markov decision process. Thus, player $2$ has a pure stationary best reply. 
Up to some subsequence, $(\la_n,x_{\la_n})_n$ is regular. 
By Proposition \ref{elputolema}, there exists $\xx\in \XX$ such that $L_\xx=L[(\la_n,x_{\la_n})_n]$. Thus, by Proposition \ref{algebra}, 
$$\lim\nolimits_{n\to \infty} \ga_{\la_n}(\om,x_{\la_n},j)=\lim\nolimits_{n\to \infty} \ga_{\la_n}(\om,\xx_{\la_n},j),
\quad \forall j\in \J^\Om.$$ On the other hand, the limit
$\lim\nolimits_{\la\to0} \ga_\la(\om,\xx_\la,j)$ exists. Consequently:
\begin{equation}\label{asymopt}
\lim\nolimits_{\la\to0} \ga_\la(\om,\xx_\la,j)=
\lim\nolimits_{n\to \infty} \ga_{\la_n}(\om,x_{\la_n},j)
\geq \limsup\nolimits_{\la\to 0}v_\la(\om),\quad \forall j\in \J^\Om.
\end{equation}
It follows that for all $\ep>0$ there exists $\la_0\in (0,1]$ such that:
\begin{equation}\label{ult}\min_{j \in \J^\Om} \ga_\la(\om,\xx_\la,j)\geq \limsup\nolimits_{\la\to 0}v_\la(\om)-\ep,\quad \forall \la\in (0,\la_0).\end{equation} 
The latter implies that 
 $v_\la(\om)\geq \limsup\nolimits_{\la\to 0}v_\la(\om)-\ep$, for all $\la\in (0,\la_0)$, and the existence of $\limla v_\la$ follows by taking the $\liminf$. The canonical strategy $\xx$ has the desired property.
   \end{proof}

\subsection{Concluding remarks}
\begin{itemize}
 \item[$(1)$] Consider an infinitely repeated stochastic game  where the past actions are observed. 
The existence of the uniform value is due to Mertens and Neyman \cite{MN81} and relies on the following result:
 \begin{theoreme}\label{mn} Let $f:(0,1)\to \RR^\Om$ be a function such that: 
\begin{itemize}
\item[$(a)$]  $\|f_\la-f_{\la'}\|\leq \int_\la^{\la'}\varphi(x)dx$, for all $0<\la<\la'<1$ and for some $\varphi\in L^1\big((0,1],\RR_+\big)$; 
\item[$(b)$] There exists $\la_0>0$ such that $\Phi(\la,f_\la)\geq f_\la$, for all $\la\in (0,\la_0)$.\footnote{$\Phi$ is the Shapley operator, defined in \eqref{shapley}.}
\end{itemize}
Then, player $1$ can guarantee $\limla f_\la$ in $\Ga_\infty$.
\end{theoreme}
One can use Theorem \ref{main} to prove the existence of the uniform value. Indeed, for any $x\in \De(\I)^\Om$, $\om\in \Om$ and $\la\in (0,1]$, let $w_\la^x(\om):=\min_{j \in \J^\Om}\ga_\la(\om,x,j)$ be the payoff guaranteed by $x$ in $\Ga_\la(\om)$.
 One can check that $w^x_\la\leq \Phi(\la,w^x_\la)$, for all $\la\in(0,1]$. Besides, for any $\xx\in \XX$, the functions $(\la \mapsto w^{\xx_\la}_\la(\om))_{\om \in \Om}$ are of bounded variation, so that player $1$ can guarantee $\limla w_\la^{\xx_\la}$ for any $\xx\in \XX$ by Theorem \ref{mn}. In particular, if $(\xx_\la)_\la$ is asymptotically optimal, player $1$ can guarantee $\limla v_\la$.
\item [$(2)$] The existence of an $\xx\in \XX$ such that $(\xx_\la)_\la$ is asymptotically optimal was
 already noticed by Solan and Vieille \cite{SV10}. The result was deduced from the semi-algebraicity of 
$\la\mapsto v_\la$, obtained in \cite{BK76} using Tarski-Seidenberg elimination theorem. 

\item [$(3)$] In the system \eqref{ns1}-\eqref{ns3} for the exponents (first part of the proof of Proposition \ref{elputolema}) note that 
all the entries of $A$ are in $\{-1,0,1\}$. This implies the existence of a solution having all its coordinates in $\{0,1/N,2/N,\dots \}$, 
 for some $N\leq|\Om||\I|^{\sqrt{|\Om||\I|}}$. 
 \item[$(4)$] Our approach fails without the finiteness assumption on $\I$, $\J$ and $\Om$.
 A recent example where $\I$ and $\J$ are compact, $q$ is continuous, $g$ is independent of the actions and the family $(v_\la)_{\la}$ does not converge is due to Vigeral \cite{vigeral12}.
\end{itemize}

\section*{Acknowledgments} I am particularly indebted to Sylvain Sorin for his careful reading and comments on 
earlier versions and to Nicolas Vieille for his valuable remarks. I would like to thank Eilon Solan for his accurate comments and remarks, and also Fabien Gensbittel, Mario Bravo and Guillaume Vigeral for the discussions at an early stage of this work.






\end{document}